\documentclass[11 pt,a4paper]{article}
\usepackage[utf8]{inputenc}
\usepackage[english]{babel}
\usepackage{amsmath,dsfont,amsthm,
enumerate,xcolor,stmaryrd,
 mathrsfs,enumitem,amssymb,
 enumitem }
 \usepackage[hidelinks]{hyperref}
\usepackage[all,knot]{xy}
\usepackage[titletoc,title]{appendix}
 \usepackage{color}
 \newcommand{\talpha}{\tilde\alpha}
 \newcommand{\txi}{\tilde\xi}
\DeclareMathOperator{\MC}{MC}
\newcommand{\ot}{\otimes}
\newcommand{\Z}{\mathbb{Z}} 
\newcommand{\Q}{\mathbb{Q}}

\newcommand{\mcc}{\mathcal{C}}

\renewcommand{\tilde}{\widetilde}

\DeclareMathOperator{\id}{id} 
 
\DeclareMathOperator{\map}{map}
\DeclareMathOperator{\aut}{aut}
\DeclareMathOperator{\ad}{ad}
\DeclareMathOperator{\Hom}{Hom}
\DeclareMathOperator{\Span}{span}
\DeclareMathOperator{\Der}{Der}

\newcommand{\la}{\langle}
\newcommand{\ra}{\rangle}

\DeclareMathOperator{\im}{im}
\renewcommand{\dots}{\ldots}

\newcommand{\C}{\mathbb{C}}
\newcommand{\h}{\mathcal{H}}

\renewcommand{\L}{\mathbb L}

\newcommand{\g}{\mathfrak g}
\renewcommand{\h}{\mathfrak h}
\newcommand{\bc}{\bar{\mathcal C}}

\newtheorem{thm}{Theorem}[section]
\newtheorem{cor}[thm]{Corollary}
\newtheorem{prop}[thm]{Proposition}
\newtheorem{claim}[thm]{Claim}
\newtheorem{lemma}[thm]{Lemma} 
\theoremstyle{definition}
\newtheorem{dfn}[thm]{Definition}
\newtheorem{exmp}[thm]{Example}
\newtheorem{rmk}[thm]{Remark}

\newcommand{\Addresses}{{
  \bigskip
  \footnotesize
  
\textsc{Department of Mathematics, Stockholm University, SE-106 91 Stockholm, Sweden}\par\nopagebreak
  \textit{E-mail addresses:} \texttt{alexb@math.su.se, bashar@math.su.se}
}}

\begin{document}
\title{A dg Lie model for relative homotopy automorphisms} 
\author{Alexander Berglund and Bashar Saleh} 
\date{}

\maketitle

\begin{abstract}
We construct a dg Lie model for the universal cover of the classifying space of the grouplike monoid of homotopy automorphisms of a space that fix a given subspace.
\end{abstract}
\section{Introduction} 
The classifying space of the monoid of homotopy automorphisms  of a space $X$ classifies fibrations with fiber homotopy equivalent to $X$. Given a subspace $A\subset X$, the classifying space of the  monoid  $\aut_A(X)$ of homotopy automorphisms that restrict to the identity on $A$  classifies all fibrations $E\to B$ with fiber homotopy equivalent to $X$ under the trivial fibration $A\times B\to B$, such that, over each $b\in B$, the canonical map from $A$ to $\im(A\to E_b)$ is a weak equivalence. The special case in which $X$ is a manifold with a non-empty boundary and where $A=\partial X$ is the boundary, has been of interest in the study of homological stability for homotopy automorphisms of manifolds (see \cite{BM12,BM14,grey}).

The main result of this paper is a proof of the following theorem:
\begin{thm}[\text{\cite[Theorem 3.4.]{BM14}}]\label{mainThm}
Let $A\subset X$ be a cofibration of simply connected spaces with homotopy types of finite CW-complexes, and let $i\colon\L_A\to \L_X$ be a cofibration that models the inclusion $A\subset X$ and where $\L_A$ and $\L_X$ are cofibrant Lie models for  $A$ and $X$ respectively. A Lie model for the universal covering of $B\aut_A(X)$ is given by the positive truncation of the dg Lie algebra of derivations on $\L_X$ that vanish on $\L_A$, denoted by $\Der(\L_X\|\L_A)\la 1\ra$.
\end{thm}
This theorem is stated in \cite{BM14} together with the suggestion that a 
proof can be given by generalizing \cite[Chapitre VII]{tanre83} , but no 
detailed proof exists in the literature. One purpose of this paper is to 
fill this gap. However, instead of following the suggested route (which 
seems to yield a rather tedious proof), we give a proof that is perhaps 
more interesting. Namely, we show that the model for relative homotopy 
automorphisms can be derived from the known model for based homotopy 
automorphisms together with general result on rational models for 
geometric bar constructions.

\subsection{Standing assumptions and notation}
\begin{itemize}
\item Throughout the paper, $X$ is a pointed simply connected space  and $A$ is a simply connected subspace that contains the basepoint of $X$. The inclusion $A\subset X$ is assumed  be a cofibration. We let $\L_A$ and $\L_X$ denote cofibrant dg Lie models over $\mathbb Q$ for $A$ and $X$ respectively. A  dg Lie algebra is cofibrant if and only if  its underlying graded Lie algebra is a free graded Lie algebra $\L(V)$ on a graded vector space $V$. We let  $i\colon\L_A\to\L_X$ denote a cofibration that  models the inclusion $A\subset X$. Recall that a map of free dg Lie algebras is a cofibration if and only if it is a  free map (see the Remark after Proposition 5.5. in \cite{quillen69}).

\item All dg Lie algebras and dg coalgebras are homologically graded, which means that that the differential lowers the degree. All  dg associative algebras are cohomologically graded. Note that if $L$ is a dg Lie algebra and $\Omega$ is a commutative dg associative algebra then $L\ot\Omega$ is a dg Lie algebra over $\Omega$ with homological grading given by
$$(L\ot\Omega)_n = \bigoplus_{p-q=n}L_p\ot \Omega^q.$$
An analogous statement holds for dg coalgebras.

\item The suspension $sV$ of a homologically graded dg vector space $V$ is a dg vector space with grading given by $(sV)_n = V_{n-1}$ and with differential given by  $d(sa)= -sd(a)$.

\item If $\h$ is a dg Lie algebra, we define its $n$-connected cover $\h\la n\ra\subseteq \h$ to be the dg Lie subalgebra given by 
$$\h\la n \ra_i=\left\{
\begin{array}{ll}
\h_i&\quad\text{if }i>n,\\
\mathrm{ker}(\h_n\xrightarrow d \h_{n-1})&\quad\text{if }i=n,\\
0&\quad\text{if }i<n.
\end{array}
\right.$$
We say that $\h$ is connected if $\h=\h\la 0\ra$ and we say that $\h$ is simply connected if $\h=\h\la 1\ra$.

\item Given a dg Lie algebra $(L,d)$, let $(\Der(L),D)$ denote the dg Lie algebra of derivations on $L$. We remind the reader that a derivation on $L$ is a linear map $\theta\colon L\to L$ that satisfy the equality $\theta[x,y] = [\theta(x),y] + (-1)^{|\theta||x|}[x,\theta(y)]$. The Lie bracket and the differential on $\Der(L)$ are given by
$$[\theta,\varphi] = \theta\circ \varphi - (-1)^{|\theta||\varphi|}\varphi\circ\theta,\qquad D(\theta)= d\circ\theta-(-1)^{|\theta|}\theta\circ d.$$

\item The connected component of the identity map in $\aut_*(X)$ and $\aut_A(X)$ are denoted by $\aut_{*,\circ}(X)$ and $\aut_{A,\circ}(X)$ respectively. The connected component of the inclusion map $\iota\colon A\hookrightarrow X$ in $\map_*(A,X)$ is denoted by $\map_*^\iota(A,X)$.
 \end{itemize}

\subsection{Strategy for the proof}\label{SubSec:Strategy}
We observe that the universal cover of $B\aut_A(X)$ is homotopy equivalent to $B\aut_{A,\circ}(X)$ (if $G$ is a topological group and $G^\circ$ is the connected component of the identity, then  $BG^\circ\to BG\to B\pi_{0}(G)\cong K(\pi_0(G),1)$ is equivalent to a fibration, giving that\break $BG^\circ\to BG$ induces  isomorphisms $\pi_k(BG^\circ)\xrightarrow\cong\pi_k(BG)$ for $k\geq 2$, which implies that $BG^\circ\simeq \widetilde{BG}$). 

In Section \ref{gbc}, we show that $B\aut_{A,\circ}(X)\simeq B(*,\aut_{*,\circ}(X),\map_*^\iota(A,X))$, where the right-hand side is the geometric bar construction of $\aut_{*,\circ}(X)$ and the left $\aut_{*,\circ}(X)$-space  $\map_*^\iota(A,X)$. The rational homotopy type of $B\aut_{*,\circ}(X)$ and $\map_*^\iota(A,X)$ are known and the identification of  $B\aut_{A,\circ}(X)$ with the geometric bar construction above gives us a way of expressing the Lie model for  $B\aut_{A,\circ}(X)$ in terms of the Lie models for $B\aut_{*,\circ}(X)$ and $\map_*^\iota(A,X)$. 

Briefly,  if a  grouplike monoid $G$ acts on $X$ from the left, then $B(*,G,X)$ is modelled by a twisted semidirect product $\g\ltimes_\xi L$ where $\g$ is a Lie model for $BG$ and $L$ is a Lie model for $X$. This is treated in Section \ref{RHTgrpAction}.

In Section \ref{der} we apply the theory of Section \ref{RHTgrpAction} to $$B\aut_{A,\circ}(X)\simeq B(*,\aut_{*,\circ}(X),\map_*^\iota(A,X)),$$ and get that a Lie model for $B\aut_{A,\circ}(X)$ is given by a twisted semidirect product  $\Der(\L_X)\la1\ra\ltimes_{\tau_*}\Hom^\tau(\bc(\L_A),\L_X)\la0\ra$ of the Lie model for $B\aut_{{*,\circ}}(X)$ and the Lie model for $\map_*^\iota(A,X)$.  We also prove that $\Der(\L_X)\la1\ra\ltimes_{\tau_*}\Hom^\tau(\bc(\L_A),\L_X)\la0\ra\simeq \Der(\L_X\|\L_A)\la1\ra$, which completes the proof of the main theorem, Theorem \ref{mainThm}.

\section{The Geometric Bar Construction}\label{gbc}
The geometric bar construction,  introduced by May (see \cite{mayiterated,may75}), is a construction that generalizes the classifying space functor. May forms a category $\mathscr K$, whose objects are triples $(X,G,Y)$, where $G$ is a topological monoid and $X$ and $Y$ are right and left $G$-spaces, respectively. A morphism between two objects, $(X,G,Y)$ and $(X',G',Y')$, in $\mathscr K$ is a triple $(i,f,j)$ where $f\colon G\to G'$ is a map of monoids, and $i\colon X\to X'$ and $j\colon Y\to Y'$ are equivariant with respect to $f$. We say that $(i,f,j)$ is a weak equivalence if $i$, $f$ and $j$ are weak equivalences, and two objects in $\mathscr K$ are called weakly equivalent if there is a zig-zag of weak equivalences connecting these two objects.  The geometric bar construction $B(X,G,Y)$ on a triple $(X,G,Y)$ in $\mathscr K$ is a topological space and defines a functor from $\mathscr K$ to the category of topological spaces.

We recall some classical facts about classifying spaces of  grouplike monoids (recall that a topological monoid $G$ is called grouplike if $\pi_0(G)$ is a group). The classifying space of a grouplike monoid $G$ is a space $BG$ that classifies all principal $G$-bundles in the following sense: The set of isomorphism classes of  principal $G$-bundles over a space $X$, is in one-to-one correspondence with the set $[X,BG]$ of homotopy classes of maps from $X$ to $BG$. 

A classifying space $BG$  of a grouplike monoid $G$ may also be recognized as a homotopy orbit space $EG\sslash G$ where $EG$ is any contractible space on which $G$ acts freely on from the left. Note that classifying spaces are only unique up to homotopy equivalences.

We also recall from \cite{may75} that the `homotopy-correct' definition of the left coset space  $G/H$ associated to an inclusion of monoids $H\subset G$ is given by 
$$G/H = B(G,H,*).$$

We  list some of the properties related to the geometric bar construction that are relevant for this paper.

\begin{prop}[\text{\cite{may75}}]\label{barProperties} Let $G$ be a topological monoid.
\begin{enumerate}[label=(\alph*)]
\item $BG=B(*,G,*)$ is a classifying space of $G$.
\item $B(X,*,*)$ is homeomorphic to $X$.
\item If $(X,G,Y)$ and $(X',G',Y')$ are weakly equivalent in $\mathscr K$ then $B(X,G,Y)$ and $B(X',G',Y')$ are weakly equivalent as spaces.
\item If $G$ is a grouplike monoid, then  $EG=B(*,G,G)$ is a contractible space on which $G$ acts freely from the right. 
\item If $G$ is a grouplike monoid, and $Y$ is a left $G$-space, then $B(*,G,Y)\simeq EG\times_G Y$.
\item If $H\subset G$ is a inclusion of grouplike monoids then $BH\simeq B(*,G,G/H)$
\end{enumerate}
\end{prop}

Applying Proposition \ref{barProperties} \textit{(f)} to $\aut_{A,\circ}(X)\subset\aut_{*,\circ}(X)$ we get that $$B\aut_{A,\circ}(X) \simeq B(*,\aut_{*,\circ}(X),\aut_{*,\circ}(X)/\aut_{A,\circ}(X))$$
(note that since $A\subset X$ is a cofibration, any homotopy automorphism of $X$ that fixes $A$ has a homotopy inverse that also fixes $A$ (see \cite[Section 6.5]{mayConcise}), which makes  $\aut_A(X)$ into a grouplike monoid).

\begin{lemma}\label{mapIsQuotient}
There is a weak equivalence of left $\aut_{*,\circ}(X)$-spaces
$$\map_*^\iota(A,X)\simeq \aut_{*,\circ}(X)/\aut_{A,\circ}(X).$$
\end{lemma}
\begin{proof} We will throughout this proof use that $X$ and $B(X,*,*)$ are interchangeable. We have that
\begin{itemize} 
\item The  map $\aut_{*,\circ}(X)\to \aut_{*,\circ}(X)/\aut_{A,\circ}(X)$ given by $$B(\id,*,*)\colon B(\aut_{*,\circ}(X),*,*)\to B(\aut_{*,\circ}(X),\aut_{A,\circ}(X),*)$$ is a quasifibration with fiber $\aut_{A,\circ}(X)$ (see \cite[Proposition 7.9]{may75}).

\item The restriction map $\mathrm{res}_A\colon\aut_{*,\circ}(X)\to\map_*^\iota(A,X)$ is a fibration, since the functor $\map_*(-,X)$ turns cofibrations into fibrations.
\item The restriction map $\mathrm{res}_A\colon\aut_{*,\circ}(X)\to\map_*^\iota(A,X)$ is invariant under the right action of $\aut_{A,\circ}(X)$ on $\aut_{*,\circ}(X)$ and therefore  the triple $$(\mathrm{res}_A,*,*)\colon (\aut_{*,\circ}(X),\aut_{A,\circ}(X),*)\to (\map_*^\iota(A,X),*,*)$$ defines a map in $\mathscr K$. Thus $$B(\mathrm{res}_A,*,*)\colon B(\aut_{*,\circ}(X),\aut_{A,\circ}(X),*)\to B(\map_*^\iota(A,X),*,*)$$ is a well-defined map.
\end{itemize}
It follows that the restriction map $\aut_{*,\circ}(X)\to \map_*^\iota(A,X)$ factors through $\aut_{*,\circ}(X)/\aut_{A,\circ}(X)$. Hence, there is a commutative diagram with rows being quasifibrations:
$$\xymatrix{\aut_{A,\circ}(X)\ar@{=}[d]\ar[r]^{incl} &\aut_{*,\circ}(X)\ar@{=}[d]\ar[r]&\aut_{*,\circ}(X)/\aut_{A,\circ}(X)\ar[d]\\
\aut_{A,\circ}(X)\ar[r]^{incl} &\aut_{*,\circ}(X)\ar[r]&\map_*^\iota(A,X)}$$
By the functoriality of the  long exact sequence of homotopy groups associated to a quasifibration and by the five lemma, it follows that there is a weak equivalence of spaces $\map_*^\iota(A,X)\simeq \aut_{*,\circ}(X)/\aut_{A,\circ}(X)$.

Moreover the restriction map $\aut_{*,\circ}(X)/\aut_{A,\circ}(X)\to \map_*^\iota(A,X)$ respects the left $\aut_{*,\circ}(X)$-action. This completes the proof.
\end{proof}

\begin{cor}\label{BautA} Let $A\subset X$ be cofibration.
There is a weak equivalence of spaces 
$$B\aut_{A,\circ}(X)\simeq B(*,\aut_{*,\circ}(X),\map_*^\iota(A,X)).$$ 
\end{cor}

\begin{proof} This is an immediate consequence of Proposition \ref{barProperties} \textit{(c), (f)} and Lemma \ref{mapIsQuotient}. \end{proof}

\begin{prop}\label{Prop:RatGBC}
The rationalization of $B\aut_{A,\circ}(X)$ is  given by
$$B\aut_{A_\Q,\circ}(X_\Q)\simeq B(*,\aut_{*,\circ}(X_\Q),\map_*^{\iota_\Q}(A_\Q,X_\Q))$$
\end{prop}
\begin{proof}
This may be obtained by `dualizing' the proof \cite[Lemma 3.1]{berglund17}.
\end{proof}

\begin{rmk}
By Proposition \ref{Prop:RatGBC} it is enough to prove the assertion in Theorem \ref{mainThm} for rational spaces $X$ and $A$ in order to get a full proof of the theorem.\end{rmk}

\section{Rational Homotopy of Grouplike Monoid Actions}\label{RHTgrpAction}

\subsection{Preliminaries: Degree-wise nilpotency and completness of dg Lie algebras}\label{Subsec:Prelim}
Nilpotent spaces are modelled by the so called degree-wise nilpotent dg Lie algebras.
\begin{dfn}
The lower central series of a dg Lie algebra $L$ is the descending filtration
$$L=\Gamma^0L\supseteq \Gamma^1L\supseteq \Gamma^2L\supseteq \cdots,$$
where $\Gamma^0L=L$ and $\Gamma^{k+1}L =[\Gamma^kL,L]$. We say that $L$ is degree-wise nilpotent if for every $n\in\Z$ there exists some $k$ such that $(\Gamma^kL)_n=0$.
\end{dfn}

\begin{dfn}
Let $\Omega_\bullet$ denote the simplicial commutative dg algebra in which $\Omega_n$ is the Sullivan-de Rham algebra of polynomial differential forms on the $n$-simplex, see \cite[Section 10 (c)]{felixrht}. The geometric realization of a degree-wise nilpotent  dg Lie algebra $L$,  is defined to be the simplicial set $\MC(L\ot\Omega_\bullet)$ of Maurer-Cartan elements of the simplicial dg Lie algebra $L\ot\Omega_\bullet$, denoted by $\MC_\bullet(L)$. We say that that a degree-wise nilpotent dg Lie algebra $L$ is a Lie model for a nilpotent space $X$ if there exists a rational homotopy equivalence
between the geometric realization $\MC_\bullet(L)$ and $X$.
\end{dfn}

In \cite{berglund15}, the geometric realization functor is extended to the so called complete dg Lie algebras.

\begin{dfn} A dg Lie algebra $\h$ equipped with a filtration $$\h = F^1\h\supseteq F^2\h\supseteq\cdots$$
is called complete if
\begin{itemize}
\item[(i)] each quotient $\h/F^i\h$ is a nilpotent dg Lie algebra, and
\item[(ii)] the canonical map  $\h\to\varprojlim\h/F^i\h$ is an isomorphism.
\end{itemize}
\end{dfn}

\begin{dfn}\label{Def:GeneralRealization}
Given a complete dg Lie algebra $\h$ we define its geometric realization to be the inverse limit 
$$\widehat\MC_\bullet(\h):=\varprojlim\MC_\bullet(\h/F^r\h).$$ We say that $\h$ is a Lie model for $X$ if the realization of $\h$ is rationally equivalent to $X$. 
\end{dfn}

\begin{rmk}
A degree-wise nilpotent dg Lie algebra $L$ together with its lower central series, makes $L$ into a complete dg Lie algebra, and we have that $\widehat\MC_\bullet(L)=\MC_\bullet(L)$. From this we may view the functor $\widehat\MC_\bullet$ as an extension of the functor $\MC_\bullet$. \end{rmk}

\begin{exmp}\label{Ex:Mapping} Let $C$ be a commutative dg coalgebra  concentrated in non-negative degrees, with coproduct $\Delta\colon C\to C\ot C$,
and let $L$ be a connected degree-wise nilpotent dg Lie algebra of finite type with Lie bracket $\ell\colon L\ot L\to L$. The convolution dg Lie algebra $\Hom(C,L)$ is a dg Lie algebra
with differential and Lie bracket given by
$$\partial(f) = d_L\circ f -(-1)^{|f|}f\circ d_C
$$
$$
[f,g] = \ell\circ f\ot g\circ \Delta.
$$
The convolution dg Lie algebra together $\Hom(C,L)$ with the filtration 
$$\Hom(C,L)\supseteq \Hom(C,L\la1\ra)\supseteq \Hom(C,L\la2\ra)\supseteq\cdots$$ is a complete dg Lie algebra. 
\end{exmp}

\subsection{Outer Actions and Exponentials}
We start by recalling some of the background for the notion of outer actions, as discussed in \cite{berglund17}. By the theory of Schlessinger-Stasheff \cite{SchlessingerStasheff} and Tanré \cite{tanre83}, we have that if $L$ is a cofibrant Lie model for $X$, then  a Lie model for the universal cover of $B\aut(X)$, or equivalently, a Lie model for $B\aut_\circ(X)$ where $\aut_\circ(X)$ is  the connected component of the identity map (see the beginning of Section \ref{SubSec:Strategy} for a motivation for this equivalence), is given by the semidirect product $\Der(L)\la1\ra\ltimes^{\ad}sL$ where $\Der(L)\la1\ra$ is the 1-connected cover of the dg Lie algebra of derivations on $L$, and $sL$ the abelian dg Lie algebra with the underlying dg vector space structure given by the suspension of $L$. The differential on the semidirect product is twisted by the adjoint map $\ad\colon sL\to \Der(L)\la1\ra$, $sl\mapsto \ad_l=[l,-]$. That is,  $\Der(L)\la1\ra\ltimes^{\ad}sL$ is a dg Lie algebra with bracket and differential given by
$$[(\theta,sx),(\varphi,sy)] = ([\theta,\varphi],(-1)^{|\theta|}s\theta(y)-(-1)^{|\varphi||x|}s\varphi(x))
$$
and
$$\partial(\theta,sx)= (D(\theta)+\ad_x,-sdx).$$

The set of homotopy classes of maps from a simply connected dg Lie algebra $\g$ to $\Der(L)\la1\ra\ltimes^{\ad}sL$  is thus in bijection with equivalence classes of $\MC_\bullet(L)$-fibrations over $\MC_\bullet(\g)$ in the category of simply connected rational spaces.
Given a map  $\psi\colon \g\to \Der(L)\la1\ra\ltimes^{\ad}sL$, the composition of $\psi$ with the projection on $\Der(L)\la1\ra$ gives a map $\g\to\Der(L)\la1\ra$ which induces a map of graded vector spaces $\alpha\colon \g\ot L\to L$ of degree 0 (this is not necessarily a chain map), and the composition of $\psi$ with the projection on $sL$ gives a map $\g\to sL$ which is equivalent to having a map $\xi\colon \g \to L$ of degree $-1$. These two maps encode a so called outer action of $\g$ on $L$.


\begin{dfn}[\text{\cite{berglund17}}]\label{DefOuterAct}
 An outer action  of $\mathfrak g$ on $L$ consists of a pair of maps $(\alpha,\xi)$, where $\alpha\colon \mathfrak g\ot L\to L$ is a map of degree 0 and $\alpha(x\ot a)$ is denoted by $x.a$, and where $\xi\colon \mathfrak g \to L$ is a map of degree $-1$, such that $\alpha$ and $\xi$ satisfy the following conditions
 \begin{itemize}
  \item[(I)] $[x,y].a  = x.(y.a)- (-1)^{|x||y|}y.(x.a)$
  \item[(II)] $x.[a,b] = [x.a,b]+(-1)^{|x||a|}[a,x.b]$
  \item[(III)] $\xi$ is a chain map, i.e. $d\xi=-\xi d$
  \item[(IV)] $\xi[x,y] = -(-1)^{|y||\xi(x)|} y.\xi(x)+ (-1)^{|x|}x.\xi(y)$
  \item[(V)] $d(x.a)=d(x).a + (-1)^{|x|}x.d(a)+[\xi(x),a]$
 \end{itemize}
\end{dfn}

\begin{prop}
Specifying an outer action of $\g$ on $L$ is tantamount to specifying
a morphism of dg Lie algebras $\g \to \Der(L)\la1\ra \ltimes^{\ad} sL$.
\end{prop}

\begin{dfn}
Given an outer action of $\g$ on $L$, the twisted semidirect product $\g\ltimes_\xi L$ of $\g$ and $L$ is a dg Lie algebra with the underlying graded vector space  given by $(\g\ltimes_\xi L)_n = \g_n\times L_n$. The Lie bracket and the differential on $\g\ltimes_\xi L$ are given  by 
$$[(x,a),(y,b)] = ([x,y], [a,b]+ x.b - (-1)^{|y||a|}y.a)$$
and
$$\partial^\xi(x,a)= (dx,da+\xi(x)).$$
\end{dfn}



Next, we associate to an outer action $(\alpha,\xi)$ of $\g$ on $L$ an action of a group $\exp_\bullet(\g)$ on the realization  $\MC_\bullet(L)$.

\begin{dfn}{\cite{berglund17}}
The exponential $\exp(\h)$ of a nilpotent Lie algebra  $\h$ concentrated in degree zero is the nilpotent group with the underlying set given by $\h$ and with multiplication given by the Campbell-Baker-Hausdorff formula. The exponential of a connected degree-wise nilpotent dg Lie algebra $\g$, $\exp_\bullet(\g)$, is defined to be the exponential  $\exp(Z_0(\g\ot \Omega_\bullet))$ of zero cycles in $\g\ot \Omega_\bullet$.
\end{dfn}

\begin{prop}\label{Prop:GaugeOuterAction}
 Let $\g$ be a simply connected dg Lie algebra and let $(\alpha,\xi)$ define an outer action of $\g$ on a dg Lie algebra $L$. The action of $\exp_\bullet(\g)$ on $\MC_\bullet(L)$ corresponding to the outer action $(\alpha,\xi)$ is given by 
 $$\exp(x).a = a +\sum_{n\geq 0}\frac{\theta_x^n (\theta_x (a)-\xi(x))}{(n+1)!}$$
where $\theta_x(a) = x.a$. 
\end{prop}

\begin{proof}
We start by recalling some of the theory of the so called gauge actions. Suppose that $\mathfrak h = \bigoplus \h_i$ is a dg Lie algebra with differential $d_\h$ and suppose that there exists some nilpotent Lie subalgebra $\h'_0\subseteq \h_0$, such that $\h$ becomes a nilpotent $\h'_0$-module (under the adjoint action). Then, there exists a group action of $\exp(\h_0')$ on $\MC(\h)$ called the gauge action, and is given by 
$$
\exp(X).A = A + \sum_{n\geq0}\frac{[X,-]^n}{(n+1)!}([X,A]-d_\h X) 
$$ where $X\in \h_0'$ and $A\in \MC(\h)$ (see \cite[Section 5.5]{manettidef} for details on the gauge action).
\\
Given a connected and bounded commutative dg algebra $\Omega=\oplus_{i=0}^n\Omega^i$, we have that $(\g\ltimes_\xi L)\ot\Omega \cong (\g\ot\Omega)\ltimes_{\xi\ot\id}(L\ot\Omega)$.
Since $\g$ is simply connected, it follows that the adjoint action of $(\g\ot\Omega)_0= (\g_1\ot\Omega^1)\oplus\cdots\oplus(\g_n\ot\Omega^n)$ on $(\g\ot\Omega)\ltimes_{\xi\ot\id}(L\ot\Omega)$ is nilpotent.
Hence the action of the subalgebra of zero cycles $Z_0(\g\ot\Omega)$ has also a nilpotent adjoint action on
$(\g\ot\Omega)\ltimes_{\xi\ot\id}(L\ot\Omega)$. Note that if $x\in Z_0(\g\ot\Omega)$ then $\partial^{\xi\ot\id}(x) =(\xi\ot\id)(x)$.
\\
Moreover, straightforward calculations give that $a\in L\ot\Omega$ is a Maurer-Cartan element in $L\ot\Omega$ if and only if it is a Maurer-Cartan element in $(\g\ot \Omega)\ltimes_{\xi\ot\id} L(\ot\Omega)$. We have that if $x\in Z_0(\g\ot\Omega)$ and $a\in L\ot\Omega$, then both $[x,a]$ and $\partial^{\xi \ot \id}(x)=(\xi\ot\id)(x)$ are elements of $L\ot\Omega\subset (\g\ot\Omega)\ltimes_{\xi\ot\id}(L\ot\Omega)$, and therefore the gauge action above defines an action of $\exp(Z_0(\g\ot\Omega))$ on $\MC(L\ot\Omega)$. In particular we have that there exists an action of $\exp_\bullet(\g)$ on $\MC_\bullet(L)$ given by the formula in the proposition.
\end{proof}

\begin{cor}\label{Cor:baseptPreserve}
If $\xi$ in the previous proposition is trivial, then the action of $\exp_\bullet(\g)$ on $\MC_\bullet(L)$ is basepoint preserving, where  $0\in\MC_\bullet(L)$ is the basepoint.
\end{cor}
\begin{proof}
This follows immediately from the explicit formula for the action, given in Proposition \ref{Prop:GaugeOuterAction}.
\end{proof}

We present some properties of $\exp_\bullet(\g)$.
\begin{prop}[\text{\cite[Corollary 3.10 and Theorem 3.15]{berglund17}}]\label{Prop:a-deloop-b-borel}
Let $\g$ be a simply connected dg Lie algebra of finite type and let $L$ be a dg Lie algebra. Suppose that $(\alpha,\xi)$ defines an outer action of $\g$ on $L$.\begin{itemize}
    \item[(a)] $\MC_\bullet(\g)$ is a delooping of $\exp_\bullet(\g)$.
    \item[(b)] The twisted semidirect product $\g\ltimes_\xi L$ is a Lie model for the Borel construction $B(*,\exp_\bullet(\g),\MC_\bullet(L))$.
\end{itemize}
\end{prop}

\subsection{Mapping spaces}\label{Subsec:MappingSpaces} 
In \cite{berglund17} it is shown that if $L$ is a connected degree-wise nilpotent dg Lie algebra of finite type, and $\Pi$ is connected dg Lie algebra then there is a weak equivalence 
$$\widehat\MC_\bullet(\Hom(\mathcal C(\Pi), L)) \to \map( \MC_\bullet(\Pi),\MC_\bullet(L)),$$
where $\mathcal C(\Pi)$ is the Chevalley-Eilenberg coalgebra construction on $\Pi$ and $\Hom(\mathcal C(\Pi), L)$ is the convolution dg Lie algebra. In particular $\Hom(\mathcal C(\Pi),L)$ is a Lie model for $\map( \MC_\bullet(\Pi),\MC_\bullet(L))$ in the sence of Definition \ref{Def:GeneralRealization}.
We want to show that this weak equivalence is equivariant with respect to the action of $\exp_\bullet(\g)$.

\begin{lemma}\label{OutActOnHom} An outer action of $\mathfrak g$ on $L$ induces an outer action of $\mathfrak g$ on the convolution dg Lie algebra $\Hom(C,L)$ for any counital cocommutative dg coalgebra $C$.
\end{lemma}
\begin{proof}
We define maps $\tilde \alpha\colon \mathfrak g\ot \Hom(C,L) \to \Hom(C,L)$ and $\tilde \xi\colon \mathfrak g \to \Hom(C,L)$. We denote $\talpha(x\ot f)$ by $x.f$ and is given by
$$\widetilde\alpha(x\ot f)(c)=(x.f)(c) = x.f(c)$$
Let $\varepsilon\colon C\to \Q$ be the counit. We define $\txi$ as the composition
$$\mathfrak g\xrightarrow{\xi} L \xrightarrow{\cong}\Hom(\Q,  L)\xrightarrow{\varepsilon^*} \Hom(C,L),$$
so $(\txi(x))(c) = \varepsilon(c)\cdot\xi(x)$.

It is straightforward to show that $\talpha$ and $\txi$ satisfies properties (I)-(V) in Definition \ref{DefOuterAct}. 
 \end{proof}
\begin{prop}\label{Prop:expEquivariance} Let $\g,\ L$, and $\Pi$ be connected nilpotent dg Lie algebras, where $L$ is of finite type. Let $(\alpha,\xi)$ be an outer action of $\g$ on $L$. The evaluation map 
 $$E\colon\MC(\Hom_{\Omega_\bullet}(\mathcal C_{\Omega_\bullet}(\Pi\ot \Omega_\bullet),L\ot \Omega_\bullet))\times \mathcal G(\mathcal C_{\Omega_\bullet}(\Pi\ot \Omega_\bullet))\to \MC(L\otimes\Omega_\bullet)$$
is $\exp_\bullet(\g)$-equivariant.
\end{prop}

\begin{proof}
That the image of $E$ really lands in  $\MC(L\otimes\Omega_\bullet)$ is proved in \cite{berglund17}. We prove that $E$ is $\exp_\bullet(\g)$-equivariant.
Using Proposition \ref{Prop:GaugeOuterAction}, we get 
\begin{align*}
E(\exp(x).(f,c)) &= E(\exp(x).f,c) =(\exp(x).f)(c)
\\&= \bigg(f + \sum_{n\geq 0}\frac{\theta^n_x(\theta_x(f)-\txi(x))}{(n+1)!}\bigg)(c)
\\ & = f(c) + \sum_{n\geq 0}\frac{\theta^n_x(\theta_x(f(c))-\txi(x)(c))}{(n+1)!}
\\ & = f(c) + \sum_{n\geq 0}\frac{\theta^n_x(\theta_x(f(c))-\varepsilon(c)\xi(x))}{(n+1)!}
\\ & = [c\text{ is a grouplike element}\Rightarrow \varepsilon(c)=1]
\\ & = f(c) + \sum_{n\geq 0}\frac{\theta^n_x(\theta_x(f(c))-\xi(x)) }{(n+1)!}
\\& = \exp(x).(f(c)) = \exp(x).E(f,c)
\end{align*}
\end{proof}

\begin{cor}
 There exists an $\exp_\bullet(\g)$-equivariant weak equivalence 
$$\widehat\MC(\Hom(\mathcal C(\Pi), L)\simeq\map(\MC_\bullet(\Pi), \MC_\bullet(L)),$$ 
that is natural in $\Pi$ and $L$.
\end{cor}
\begin{proof}
By Proposition \ref{Prop:expEquivariance} we have that the adjoint map of $E$
$$ \MC(\Hom_\Omega(\mathcal C_\Omega(\Pi\ot \Omega_\bullet),L\ot \Omega_\bullet))\to \map(\MC_\bullet(\Pi),\MC_\bullet(L))$$
is $\exp_\bullet(\g)$ equivariant. By \cite[Theorem 3.16]{berglund17} this is also a weak equivalence that is natural in $\Pi$ and $L$. Following \cite[Theorem 3.17]{berglund17}, there exists a natural isomorphism 
$$\widehat\MC_\bullet(\Hom(\mathcal C(\Pi),L)\cong \MC(\Hom_\Omega(\mathcal C_\Omega(\Pi\ot \Omega_\bullet),L\ot \Omega_\bullet)).$$
It is straightforward to show that this isomorphism respects the $\exp_\bullet(\g)$-action.
\end{proof}

\begin{cor}\label{Cor:ModelBasedMaps}
There is a weak equivalence of spaces
$$\widehat\MC_\bullet(\Hom(\bc(\Pi),L))\to\map_*(\MC_\bullet(\Pi),\MC_\bullet(L)).$$
Moreover, for every outer $(\alpha,\xi)$ of $\g$ on $L$ where $\xi$ is trivial,  the map
above is $\exp_\bullet(\g)$-equivariant.
\end{cor}

\begin{proof} By \cite[Proposition 5.4]{berglund15}, the functor $\widehat\MC_\bullet$ takes surjections of complete dg Lie algebras to (Kan) fibrations. In particular, the surjection $\Hom(\mathcal C(\Pi),L)\to \Hom(\Q,L)\cong L$ induces a fibration $\widehat\MC_\bullet(\Hom(\mcc(\Pi),L) \to\MC_\bullet(L)$, which has fiber $\widehat\MC_\bullet(\Hom(\bc(\Pi),L)$.

Moreover, the map $\map(\MC_\bullet(\Pi),\MC_\bullet(L))\to \map(*,\MC_\bullet(L))\cong \MC_\bullet(L)$ is a fibration, which has fiber $\map_*(\MC_\bullet(\Pi),\MC_\bullet(L))$, and thus we get a commuting diagram 
$$\xymatrix{\widehat\MC_\bullet(\Hom(\bc(\Pi),L))\ar[r]\ar[d]&\widehat\MC_\bullet(\Hom(\mcc(\Pi),L))\ar[r]\ar[d] & \MC_\bullet(L)\ar@{=}[d]
\\
\map_*(\MC_\bullet(\Pi),\MC_\bullet(L))\ar[r]&\map(\MC_\bullet(\Pi),\MC_\bullet(L))\ar[r]& \MC_\bullet(L)}
$$ with rows being fibrations. The long exact sequence of homotopy groups yields now the weak equivalence $\widehat\MC_\bullet(\Hom(\bc(\Pi),L))\to \map_*(\MC_\bullet(\Pi),\MC_\bullet(L))$. This completes the proof for the first part of the statement.

For the second part, we just recall  that the triviality of $\xi$ gives that the induced $\exp_\bullet(\g)$-action on $\MC_\bullet(L)$ is basepoint preserving, see Corollary \ref{Cor:baseptPreserve}, and will therefore induce an action on the based mapping space $\map_*(\MC_\bullet(\Pi),\MC_\bullet(L))$. 
\end{proof}

\begin{prop}\label{Prop:ModelBorel}
Let $\g= \Der(L)\la1\ra$. There exists an outer action $(\alpha,\xi)$ of $\g$ on $L$ where $\alpha(\theta,x) =\theta(x)$ and where $\xi=0$. The action of  $\exp_\bullet(\g)$ on $\MC_\bullet(L)$ yields a map $\exp_\bullet(\g)\to \aut_{\circ,*}(\MC_\bullet(L))$ which is a weak equivalence.    In particular the triples $(*,\exp_\bullet(\g), \widehat\MC_\bullet(\Hom(\bc(\Pi),L))$ and $(*,\aut_{*,\circ}(\MC_\bullet(L)),\map_*(\MC_\bullet(\Pi),\MC_\bullet(L)))$ are weakly equivalent in the category $\mathscr K$ (discussed in Section \ref{gbc}). A Lie model for 
 $$B(*,\aut_{*,\circ}(\MC_\bullet(L)),\map_*(\MC_\bullet(\Pi),\MC_\bullet(L)))$$
 is given by $\g\ltimes\Hom(\bc(\Pi),L)$.
\end{prop}

\begin{proof}
 It follows by the theory  of Schlessinger-Stasheff \cite{SchlessingerStasheff} and Tanré \cite{tanre83} that if $L$ is a cofibrant dg Lie algebra, then  a Lie model for $B\aut_{*,\circ}(\MC_\bullet(L))$ is given by $\Der(L)\la1\ra$. By Proposition \ref{Prop:a-deloop-b-borel} \textit{(a)} it follows that $\exp_\bullet(\g)$ is weakly equivalent to $\aut_{*,\circ}(\MC_\bullet(L))$. This fact, together with Corollary \ref{Cor:ModelBasedMaps}, gives the equivalence of triples mentioned in the proposition. 
 \\
The statement regarding the Lie model is a consequence of Proposition \ref{Prop:a-deloop-b-borel} \textit{(b)}.
 \end{proof}

\section{Modelling Homotopy Automorphisms with Derivations}\label{der}
The ultimate goal of this paper is to study the rational homotopy  of $$B\aut_{A,\circ}(X)\simeq B(*,\aut_{*,\circ}(X),\map_*^\iota(A,X)),$$
which is a connected component in $ B(*,\aut_{*,\circ}(X),\map_*(A,X))$. The disconnected space $ B(*,\aut_{*,\circ}(X),\map_*(A,X))$ is modelled by the complete dg Lie algebra $$\Der(\L_X)\la1\ra\ltimes\Hom(\bc(\L_A),\L_X)$$
(see Proposition \ref{Prop:ModelBorel}). If $\h$ is a complete dg Lie algebra model for a disconnected space $W$, one may extract a dg Lie algebra model for a connected component $W^\tau\subset W$ by the following proposition:

\begin{prop}[\text{\cite[Theorem 5.5]{berglund15}}]\label{Lemma:TwistedLie} Let $\h$ be a complete dg Lie algebra and let $\tau$ be a Maurer-Cartan element in $\h$. The connected component of $\widehat \MC_\bullet(\h)$ that contains  $\tau$ is weakly equivalent to $\MC_\bullet(\h^\tau\la0\ra)$ where $\h^\tau$ is the dg Lie algebra whose underlying graded Lie algebra structure coincides with the one of $\h$ but with a twisted differential $\partial^\tau = \partial + \ad_\tau$. 
\end{prop}

We apply this proposition in order to get a Lie model for $B\aut_{A,\circ}(X)$:
\begin{prop}\label{Prop:TauTwist}
Let $\tau\in\Hom(\bc(\L_A),\L_X)$ be the Maurer-Cartan element given by the composition
$$\tau\colon\bc(\L_A) \xrightarrow{\pi_A}\L_A\xrightarrow{i}\L_X,$$
where $\pi_A\colon \bc(\L_A) \to\L_A$ is the universal twisting morphism and $i\colon \L_A\to\L_X$  is a cofibration that model the inclusion $\iota\colon A\hookrightarrow X$. A Lie model for 
$$B(*,\aut_{*,\circ}(X),\map_*^\iota(A,X))$$
is given by  
\begin{equation}\label{Formula:Twists}
\left(\Der(\L_X)\la1\ra\ltimes \Hom(\bc(\L_A),\L_X)\la0\ra\right)^{(0,\tau)}= \Der(\L_X)\la1\ra\ltimes_{\tau_*}\Hom^\tau(\bc(\L_A),\L_X)\la0\ra,
\end{equation}
where $\tau_*(\theta)= -(-1)^{|\theta|}\theta\circ \tau$.
\end{prop}

\begin{proof}
The connected component of the Maurer-Cartan element $$(0,\tau)\in \Der(\L_X)\la1\ra\ltimes \Hom(\bc(\L_A),\L_X)$$ in the realization corresponds to the connected component $B(*,\aut_{*,\circ}(X),\map_*^\iota(A,X))$ in $B(*,\aut_{*,\circ}(X),\map_*(A,X))$. By Proposition \ref{Lemma:TwistedLie}, a Lie model for $$B(*,\aut_{*,\circ}(X),\map_*^\iota(A,X))$$ is given by the left hand-side of \eqref{Formula:Twists}. The equality \eqref{Formula:Twists}  may now be checked by hand.
\end{proof}

\begin{dfn}
Let $f\colon \h\to \Pi$ be a morphism of dg Lie algebras, define $\Der_f(\h,\Pi)$ to be the dg vector space of so called $f$-derivations from $\h$ to $\Pi$. An $f$-derivation is a linear map $\theta: \h\to \Pi$ that satisfies 
$$\theta[x,y] = [\theta(x),f(y)]+(-1)^{|\theta||x|} [f(x),\theta(y)].$$ 
\end{dfn}

\begin{prop}
Let $\tau\colon \bc\L_A\to\L_X$ be as in Proposition \ref{Prop:TauTwist}. The map 
$$s\pi_A^*\colon\Der_i(\L_A,\L_X)\la1\ra\to s(\Hom^\tau(\bc\L_A,\L_X)\la0\ra)$$  given by
$s\pi_A^*(\theta) = (-1)^{|\theta|+1}s(\theta\circ\pi_A)$ is a quasi-isomorphism.
\end{prop}

\begin{proof}
It is enough to show that $\pi_A^*\colon \Der_i(\L_A,\L_X)\la1\ra\to \Hom^\tau(\bc\L_A,\L_X)\la0\ra$, $\pi_A^*(\theta)=\break (-1)^{|\theta|+1}\theta\circ\pi_A$ induces isomorphisms in shifted homology, i.e.
$$H(\pi_A^*)\colon H_{p+1}( \Der_i(\L_A,\L_X)\la 1\ra)\xrightarrow \cong H_p(\Hom^\tau(\bc\L_A,\L_X)\la 0\ra).$$
We will use that universal twisting morphism $\pi_A$ satisfies the Maurer-Cartan equation 
\begin{equation}\label{MCeqn}\pi_A\circ d = -d\circ\pi_A-\frac12[\pi_A,\pi_A].\end{equation}
We have that
\begin{align*}
D(\pi_A^*(\theta))& = (-1)^{|\theta|+1}d\circ\theta\circ \pi_A -\theta\circ \pi_A\circ d + (-1)^{|\theta|+1}[\tau,\theta\circ\pi_A]
\\ & \overset{\eqref{MCeqn}}=(-1)^{|\theta|+1} d\circ\theta\circ \pi_A+ (-1)^{|\theta|+1}[i\circ\pi_A,\theta\circ \pi_A]-\theta\circ\big(-d\circ\pi_A-\frac12[\pi_A,\pi_A]\big)
\\& = [\theta \text{ is an $i$-derivation.}]
\\ & = (-1)^{|\theta|+1}d\circ\theta\circ \pi_A+ (-1)^{|\theta|+1}[i\circ\pi_A,\theta\circ \pi_A]+\theta\circ d\circ\pi_A\\&\qquad+\left(\frac12[\theta\circ \pi_A,i\circ \pi_A] +(-1)^{|\theta|} \frac12[i\circ \pi_A,\theta\circ \pi_A]
\right)
\\ & = (-1)^{|\theta|+1}d\circ\theta\circ \pi_A+ (-1)^{|\theta|+1}[i\circ\pi_A,\theta\circ \pi_A]+\theta\circ d\circ\pi_A+(-1)^{|\theta|}[i\circ\pi_A,\theta\circ\pi_A]
\\ & = (-1)^{|\theta|+1}d\circ\theta\circ \pi_A+\theta\circ d\circ\pi_A
\\& =- \pi_A^*(D(\theta)),
\end{align*}
proving that $\pi_A^*$ is a chain map.

Now we prove that $\pi_A^*$ is a quasi-isomorphism (up to a degree shift). If $L$ is a connected dg Lie algebra, let $Q(L)=L/[L,L]$ denote the chain complex of the indecomposable elements in $L$.
\begin{lemma}[\text{\cite[Proposition 22.8]{felixrht}}]\label{CEindec} The composition
$$\bc\L_A\xrightarrow{\pi_A}\L_A\to Q(\L_A)$$
induces isomorphisms in homology $$H_{p+1}(\bc\L_A)\xrightarrow\cong H_{p}(Q(\L_A)).$$
\end{lemma}

Now we consider the complete filtration $\Der_i(\L_A,\L_X)\la 1\ra = F^1\supseteq F^2\supseteq\cdots$, where $F^p$ is the subcomplex of $i$-derivations that vanish on elements of degree $<p$, and the complete filtration $\Hom^\tau(\bc \L_A,\L_X)\la 0\ra = \hat F^1\supseteq \hat F^2\supseteq\cdots$ where $\hat F^p$ is the the subcomplex of linear maps that vanish on elements of degree $<p+1$. 

With respect to these filtrations, $\pi_A^*$ becomes a map of filtered complexes and induces a map of spectral sequences. We have that the first filtration gives rise to a first quadrant spectral sequence with $E_2$-term
$$E_2^{p,-q} = \Hom(H_p(Q(\L_A)),H_q(\L_X)),$$
and the second filtration gives rise to a first quadrant  spectral sequence with $\hat E_2$-term
$$\hat E_2^{p+1,-q}=\Hom(H_{p+1}(\bc\L_A),H_q(\L_X)).$$

By Lemma \ref{CEindec}, the induced map $E_2(\pi_A^*)\colon E_2^{p,-q}\to\hat E_2^{p+1,-q}$ is an isomorphism. The comparison theorem (see for instance \cite{weibel}) gives now that $\pi_A^*$ is indeed a quasi-isomorphism up to a degree shift.
\end{proof}

We are left to show that there exists a weak equivalence of dg Lie algebras
$$\Der(\L_X\|\L_A)\la1\ra\simeq \Der(\L_X)\la1\ra\ltimes_{\tau_*}\Hom^\tau(\bar{\mathcal C}\L_A,\L_X)\la0\ra$$
in order to complete the proof of the main theorem, Theorem \ref{mainThm}.

\begin{prop}\label{Prop:Zeta} Let $i\colon \L_A\to\L_X$ be a cofibration (i.e. a free map). Then  we may view $\L_A$ as a subalgebra of      $\L_X$. Let $\Der(\L_X\|\L_A)$ be the dg Lie algebra of derivations on $\L_X$ that vanish on $\L_A$. The map
$$\zeta\colon\Der(\L_X\|\L_A)\la1\ra\to \Der(\L_X)\la1\ra\ltimes_{\tau_*}\Hom^\tau(\bar{\mathcal C}\L_A,\L_X)\la0\ra$$
given by inclusion into the first term is a quasi-isomorphism of dg Lie algebras.
\end{prop}

\begin{proof} It is straightforward to show that $\zeta$ is a map of Lie algebras. We want to show that $\zeta$ is a chain map. This is equivalent to having that $\tau_*|_{\Der(\L_X\|\L_A)\la1\ra} = 0$.
We have that $\tau\colon \bc(\L_A)\to \L_X$ factors through $\bc(\L_X)$
$$\xymatrix{\bc(\L_A)\ar[rr]^\tau\ar[rd]_{\bc(i)}&&\L_X\\
& \bc(\L_X)\ar[ru]_{\pi_X}}$$
where $\pi_X\colon \bc(\L_X)\to \L_X$ is the universal twisting morphism. Given a derivation $\theta\in\Der(\L_X)$, it induces a coderivation $\Theta\in \mathrm{Coder}(\bc(\L_X))$ given by $$\Theta(sx_1\wedge\cdots\wedge sx_k) = \sum\pm sx_1\wedge\cdots\wedge s\theta(x_i)\wedge\cdots\wedge sx_k$$ 
so that $\pi_X\circ \Theta =(-1)^{|\theta|}\theta\circ \pi_X$.
Now assume that $\theta\in\Der(\L_X\|\L_A)\la1\ra$, i.e. $\theta\circ i = 0$, then we have that $\Theta\circ\bc(i) = 0$ and in particular
\begin{equation*}
\tau_*(\theta) = \theta\circ \tau = \theta\circ \pi_X\circ \bc(i)= (-1)^{|\theta|}\pi_X\circ\Theta\circ\bc(i) = 0
\end{equation*}
This gives that $\zeta\colon \Der(\L_X\|\L_A)\la1\ra\to  \Der(\L_X)\la1\ra\ltimes_{\tau_*}\Hom^\tau(\bar{\mathcal C}\L_A,\L_X)\la0\ra$ is a chain map, and therefore also a map of dg Lie algebras.

Now we show that $\zeta$ is a quasi-isomorphism. In the model category of chain complexes we have that the homotopy cofiber of the projection map
$$\rho\colon \Der(\L_X)\la1\ra\ltimes_{\tau_*}\Hom^\tau(\bc(\L_A,\L_X)\la0\ra\to \Der(\L_X)\la1\ra$$
is the mapping cone $\Der(\L_X)\la1\ra\oplus s(\Der(\L_X)\la1\ra\ltimes_{\tau_*}\Hom^\tau(\bc(\L_A),\L_A)\la0\ra)$, denoted by $cone(\rho)$, equipped with the differential given by $d(\theta,s\psi,s\eta)=(d\theta-\psi,-s\partial^\tau(\psi,\eta))$.
In particular we have that 
$$\Der(\L_X)\la1\ra\ltimes_{\tau_*}\Hom^\tau(\bc(\L_A,\L_X)\la0\ra\xrightarrow{\rho} \Der(\L_X)\la1\ra\to cone(\rho)$$
is equivalent to a homotopy cofibration.

Moreover we have that there is a short exact sequence of chain complexes
$$0\to s(\Hom^\tau(\bc(\L_A),\L_A))\la0\ra)\xrightarrow{incl}cone(\rho)\to \Der(\L_X)\la1\ra\oplus s\Der(\L_X)\la1\ra\to0$$

We have that $\Der(\L_X)\la1\ra\oplus s(\Der(\L_X)\la1\ra)\simeq 0$ (since $\Der(\L_X)\la1\ra\oplus s(\Der(\L_X)\la1\ra$ is the mapping cone on the identity map), so it follows that $$s(\Hom^\tau(\bc(\L_A),\L_A)\la 0\ra)\xrightarrow{incl}cone(\rho)$$ is a homotopy equivalence. 
It follows now that the composition of homotopy equivalences
$$\alpha\colon \Der_i(\L_A,\L_X)\la1\ra\xrightarrow{s\pi_A^*}s(\Hom^\tau(\bc(\L_A),\L_X)\la0 \ra)\xrightarrow{incl} cone(\rho)$$ is a homotopy equivalence.
Since $i\colon \L_A\to \L_X$ is a free map, the restriction map $\Der(\L_X)\la1 \ra\to \Der_i(\L_A,\L_X)\la1 \ra$ is onto with kernel $\Der(\L_X\|\L_A)\la 1\ra$. In particular we have a short exact sequence 
$$0\to \Der(\L_X\|\L_A)\la 1\ra\to\Der(\L_X)\la1 \ra\to \Der_i(\L_A,\L_X)\la 1\ra\to 0.$$
Now consider the following (non-commuting) diagram with rows being homotopy cofibrations
\begin{equation}\label{diagram}
\xymatrix{
\Der(\L_X\|\L_A)\la 1\ra\ar[r]\ar[d]_{\zeta} & \Der(\L_X)\la 1\ra\ar[r]\ar@{=}[d] & \Der_i(\L_A,\L_X)\la 1\ra\ar[d]^{-\alpha}_{\wr} \\
\Der(\L_X)\la1 \ra\ltimes_{\tau_*}\Hom^\tau(\bar{\mathcal C}\L_A,\L_X)\la0 \ra\ar[r] & \Der(\L_X)\la 1\ra\ar[r] & cone(\rho)
}\end{equation}

\begin{claim}
 The diagram above commutes up to homotopy
\end{claim}
\begin{proof}
The left square commutes strictly, so we are left to show that the right square commutes up to homotopy.
In other words, we want to show that 
$$\Phi,\Psi\colon \Der \L_X\la 1\ra\to cone(\rho) =  \Der(\L_X)\la1 \ra\oplus s(\Der(\L_X)\la1 \ra\ltimes_{\tau_*}\Hom^\tau(\bc(\L_A),\L_A)\la 0\ra)$$
given by $\Phi(\theta) = (\theta,0,0)$ and $\Psi(\theta) = (0,0,-s\pi_A^*(\theta\circ i))$
are homotopic.

Let $H\colon\Der(\L_X)\la1 \ra\to \Der(\L_X)\la1 \ra\oplus s(\Der(\L_X)\la1 \ra\ltimes_{\tau_*}\Hom^\tau(\bc(\L_A),\L_A)\la0 \ra)$ be given by 
$$H(\theta)=(0,s\theta,0)$$
We have that 
$$(dH+Hd)(\theta) = d(0,s\theta,0)+H(d(\theta))$$
$$= (-\theta,-sd(\theta),-s\tau_*(\theta))+(0,sd(\theta),0)=(-\theta,0,-s\tau_*(\theta))$$
which is $(\Psi-\Phi)(\theta)$.
\end{proof}
Now as we have that \eqref{diagram} commutes up to homotopy, it induces a (strict) map of the long exact sequences associated to the cofibrations. Since $\alpha\colon \Der_i(\L_A,\L_X)\la 1\ra\to cone(\rho)$ and $\id\colon \Der(\L_X)\la1 \ra\to\Der(\L_X)\la 1\ra$ induces isomorphisms in homology, it follows by the five lemma that $\zeta$ also induces isomorphisms in homology. 
\end{proof}

\section{Examples}
\begin{exmp} A Lie model for $\C P^k$, $k\geq 1$ is given by the cofibrant dg Lie algebra $\L(x_1,\dots,x_{k})$ on the free graded vector space $\Span_\Q(x_1,\dots, x_{k})$ where $|x_i|=2i-1$ and where the differential is given by $d(x_i) =\frac12\sum_{p+q=i}[x_p,x_q]$ (see \cite[§24.(f)]{felixrht}).
The  inclusion of $\C P^k\to \C P^n$, $1\leq k<n$ is modelled by the  free map induced by the inclusion
$$\Span_\Q(x_1,\dots,x_{k})\hookrightarrow \Span_\Q(x_1,\dots,x_{n}).$$

In particular we have that the underlying dg vector space of $\Der(\L_{\C P^n}\|\L_{\C P^k})$ is isomorphic to $$\Hom(\Span_\Q(x_{k+1},\dots,x_n),\L(x_1,\dots,x_n)),$$
where the differential on a map $f$ of homogeneous degree $|f|$ is given by 
$$d(f)(x_q)= d\circ f(x_q)-(-1)^{|f|} \frac12\sum_{i+j=q}[f(x_i),x_j]-\frac12\sum_{i+j=q}[x_i,f(x_j)],$$
where we set $f(x_i)=0$ if $i\leq k$. We observe that if $n=k+1$ then only first term in the differential above survives. In this particular case we have an isomorphism of chain complexes
$$\Der(\L_{\C P^{k+1}}\|\L_{\C P^k}) = s^{-2k-1}\L_{\C P^{k+1}}.$$
Hence, we see that
$$\pi_{*+2k+1}^\Q(B\aut_{\C P^k}(\C P^{k+1}))=\pi_*^\Q(\C P^{k+1}).$$
\end{exmp}

\begin{exmp} Every simply connected topological space $X$ admits a minimal Lie model of the form $(\L(s^{-1}\tilde H_*(X;\Q)),d)$, where the generating vector space is the desuspension of the reduced rational homology of $X$ (see \cite[Chapter 24]{felixrht}). In particular, a minimal Lie model for the sphere $S^{n-1}$ is given by $\L(u)$ where $|u|=n-2$.

If $X$ is an $n$-dimensional simply connected compact manifold with boundary $\partial X \cong S^{n-1}$, then for every basis $B=\{x_1,\dots,x_m\}$ of $s^{-1}\tilde H_*(X;\Q)$ there exists a `dual basis' $B^{\#}=\{x_1^{\#},\dots,x_m^{\#}\}$ such that $|x_i^{\#}|+|x_i|=n-2$ and such that the inclusion $S^{n-1}\to X$ is modelled by 
the dg Lie algebra morphism
$$\L(u)\to \L(s^{-1}\tilde H_*(X;\Q)),\qquad u\mapsto \omega=\frac12\sum_{i=1}^m[x_i^{\#},x_i]$$
(see \cite{stasheff83} and \cite[\S3.5]{BM14} for details). 

Note that the dg Lie algebra map above is not a cofibration. In order to be able to apply Theorem \ref{mainThm}, one needs to replace the map by a cofibration.
This can be done by adding generators $u$ and $v$ to the Lie model of $X$ and define $d(u) =0$ and $d(v) = u-\omega$. Then the inclusion
$\L(u)\to \L(s^{-1}\tilde H_*(X),u,v)$ is a cofibration that models the inclusion. Theorem \ref{mainThm} shows that $\Der(\L(s^{-1}\tilde H_*(X),u,v)\| u)$ is a Lie model for the universal cover of $B\aut_\partial(X)$. In \cite{BM14}, the authors go further and show that $\Der(\L(s^{-1}\tilde H_*(X),u,v)\| u)$ is quasi-isomorphic to $\Der(\L(s^{-1}\tilde H_*(X))\| \omega)$. However, in general, if $j\colon \L_A\to \L_X$ models a cofibration $A\subset X$ where $\L_A$ and $\L_X$ are cofibrant dg Lie algebras, but where $j$ is not a cofibration, then it is not necessarily true that the Lie subalgebra of $\Der(\L_X)$ of derivations that vanish on the image of $j$ is a Lie model for the universal cover of $B\aut_A(X)$. We will see this in the next example.
\end{exmp}

\begin{exmp}
In Theorem \ref{mainThm} it is required that the Lie algebra map $i\colon\L_A\to\L_X$  that models the inclusion $A\subset X$ is a cofibration. In this example we show that this condition is necessary.

Consider the inclusion $S^3\subset D^4$. A cofibration between cofibrant dg Lie algebras  that models the inclusion is given by 
$$(\L(u),|u|=2) \to (\L(u,v),\ dv=u,\ |u|=2,\ |v|=3).$$ Hence we know that a Lie model for $B\aut_{S^3,\circ}(D^4)$ is given by $\Der(\L(u,v)\|\L(u))$ which one easily shows is homotopically trivial.  Let us now model the inclusion  $S^3\subset D^4$ by a Lie map which is not a cofibration. We let a cofibrant model for $S^3$ be given by the abelian dg Lie algebra
$$\L_{S^3} = (\L(u),\ |u|=2, du=0).$$ Since $D^4$ is contractible, any homotopically trivial dg Lie algebra is a Lie model for $D^4$, and any map from  $\L_{S^3}$ to that Lie model of $D^4$ is a model for the inclusion $S^3\subset D^4$. We let 
$$\L_{D^4} = (\L(a,b),\ |a|=1,\ |b|=2,\ db=a)$$
be Lie model for $D^4$ and we let the inclusion $S^3\subset D^4$ be modelled by the map $i\colon \L(u)\to \L(a,b)$, $i(u)=[a,a]$. Now we show that 
$\Der(\L(a,b)\|[a,a])$ is not weakly equivalent to the trivial dg Lie algebra.
Let $\ad_a\in \Der(\L(a,b)\|[a,a])$ be given by $\ad_a(x)= [a,x]$ ($\ad_a$ vanishes on $[a,a]$ since $[a,[a,a]]=0$ by the graded Jacobi identity). Straightforward calculations give that $\ad_a$ is a cycle in $\Der(\L(a,b)\|[a,a])$. Now we show that $\ad_a$ is  not a boundary in $\Der(\L(a,b)\|[a,a])$. 
We have that any $g\in \Der(\L(a,b))$ of degree 2 is determined by its images on the  generators. We have that $g(a) = \alpha[b,a]$  for some $\alpha\in\Q$ since the degree three part of $\L(a,b)$ is spanned by $[b,a]$. We also have that  $g(b) = \beta [b,[a,a]]$ for some $\beta\in\Q$ since the degree four part of $\L(a,b)$ is spanned by  $[b,[a,a]]$.
Solving the equation $Dg=\ad_a$ gives that $\alpha=1$ and that $\beta$ can be chosen arbitrary. However, we get that 
$$g[a,a] = 2[[b,a],a]\neq 0$$
showing that $g\not\in \Der(\L(a,b)\|[a,a])$, so $\ad_a$ is not a boundary in $\Der(\L(a,b)\|[a,a])$. We conclude that $\Der(\L(a,b)\|[a,a])$ is not homotopically trivial, and therefore not a Lie model for $B\aut_{S^3,\circ}(D^4)$.

\end{exmp}

\bibliographystyle{amsalpha}
\bibliography{references}
\noindent
\Addresses

\end{document}